\documentclass[canadian,a4paper,?,12pt,twoside,openright,verbose,reqno,final]{amsart}

\usepackage[canadian]{babel}

\usepackage[ascii]{inputenc} 

\usepackage{exscale} 

\usepackage{amsthm,latexsym,amssymb,amsbsy} 

\usepackage{mathtools} 

\usepackage{mathrsfs,dsfont,bold-extra} 

\usepackage[mathscr]{eucal} 

\usepackage[poorman,sloped]{fourier} 

\usepackage{hyperref}

\hypersetup{%
unicode,%
setpagesize=true,%
pdftitle={Representations of Hermitian Commutative {\9047\073}-Algebras by Unbounded Operators},%
pdfauthor={\copyright\ Marco Ch. R. Thill},%
pdfsubject={spectral theory of representations},%
pdfcreator={AMS-LaTeX},%
pdfkeywords={spectral theorem, representation, unbounded operator, Hermitian *-algebra},%
pdfpagemode={UseNone},%
pdfstartpage={1},%
pdfcenterwindow={true},%
pdfdisplaydoctitle={true},%
pdfpagelayout={TwoPageRight},
pdfstartview={Fit},%
pdfview={Fit},%
bookmarksopen={true},%
bookmarksopenlevel={0},%
bookmarksnumbered={true},%
pageanchor={true},%
pdfnewwindow={true},%
}

\newtheorem{definition}{\textsc{Definition}}
\newtheorem{proposition}[definition]{\textsc{Proposition}}
\newtheorem{theorem}[definition]{\textsc{Theorem}}
\newtheorem{corollary}[definition]{\textsc{Corollary}}
\newtheorem{reminder}[definition]{\textsc{Reminder}}
\newtheorem{lemma}[definition]{\textsc{Lemma}}
\newtheorem{statement}[definition]{\textsc{Statement}}
\newtheorem{observation}[definition]{\textsc{Observation}}

\newcommand{\st}{\ensuremath{\ast}}

\newcommand{\slangle}{\langle\,}

\newcommand{\srangle}{\,\rangle}

\newcommand{\wht}{\widehat}

\newcommand{\diff}{\mathrm{d}}

\newcommand{\id}{\mathrm{id}}

\newcommand{\sa}{_{\textstyle{\ensuremath{\mathrm{sa}}}}}

\newcommand{\s}{\mathrm{spec}}

\newcommand{\iu}{\mathrm{i}} 

\newcommand{\range}{\mathrm{ran}}

\begin{document}

\title[Representations of Hermitian Commutative \st-Algebras \ldots]%
{Representations of Hermitian Commutative \st-Algebras by Unbounded Operators}%

\author[M. Thill]{Marco Ch.\ R.\ Thill}


\dedicatory{Dedicated to Christian Berg on the occasion of his sixty-fifth birthday}







\thanks{Many thanks to K.\ Schm\"{u}dgen for kind correspondence and
reference to \cite{SchIn}. Many thanks go to Torben Maack Bisgaard as well as other.}

\begin{abstract}
We give a spectral theorem for unital representations of
Hermitian commutative unital \st-algebras by possibly
unbounded operators on pre-Hilbert spaces. A better
result is known for the case in which the \st-algebra
is countably generated.

\end{abstract}

\maketitle


\section{\textbf{Statement of the Main Result}}

Our main result is the following theorem:

\begin{theorem}\label{main}%
Let $\pi$ be a unital representation of a Hermitian commutative unital
\st-algebra $A$ on a pre-Hilbert space $H \neq \{ 0 \}$.

The operators $\pi(a)$ $(a \in A)$ are essentially normal in the completion.

There is a spectral measure $P$ on a subset of $\Delta^*(A)$, acting on the completion of $H$,
such that the closure of the operator $\pi(a)$ in the completion of $H$ is given by
\[ \hspace{-9em} (i) \qquad \qquad \qquad \qquad
 \overline{\pi(a)} = \int \widehat{a} \:\diff P \qquad (a \in A). \]

We then have that
\begin{itemize}
 \item[$(ii)$] the spectral resolution of the
 normal operator $\overline{\pi(a)}$ with $a \in A$ is the image
 \,$\widehat{a} \,(P) = P \circ {\widehat{a}}^{\,-1}$ of $P$ under the function \,$\widehat{a}$,
\item[$(iii)$] so the normal operators $\overline{\pi(a)}$ $(a \in A)$ commute spectrally (strongly),
\item[$(iv)$] a bounded operator $b$ on the completion of $H$ commutes with the
normal operators $\overline{\pi(a)}$ $(a \in A)$ if and only if $b$ commutes with $P$,
\end{itemize}
\end{theorem}
Notation and terminology are explained in the following section.
A proof is given in section \ref{mainsect} below.
For countably generated \st-algebras, a better result is available,
cf.\ Savchuk and Schm\"{u}dgen \cite[Theorem 7 p. 60 f.]{SchIn}.
Cf.\ Schm\"{u}dgen \cite[Theorem 7.23 p.\ 148 f.]{SchI} \& \cite[p.\ 237]{Schm};
also see \cite{BereI}, \cite{BereII}, \cite{Sam} for related results.%
\pagebreak

\section{\textbf{Terminology and Notation}}

\subsection{Representations}

Let $( H , \langle \cdot , \cdot \rangle )$ be a pre-Hilbert space. Denote by
$\mathrm{End}(H)$ the algebra of linear operators mapping $H$ to $H$.
Let $A$ be some \st-algebra. By a representation of $A$ on $H$ we shall
understand an algebra homomorphism $\pi$ from $A$ to $\mathrm{End}(H)$
such that
\[ \langle \pi (a) x, y \rangle = \langle x, \pi (a^*) y \rangle
\qquad \text{for all} \quad a \in A, \quad \text{and all} \quad x, y \in H. \]
The range $\range(\pi) \vcentcolon= \pi(A)$ then is a \st-algebra as is easily seen.
(Please note that $\mathrm{End}(H)$ need not be a \st-algebra for this.)

\subsection{Unitisation and Spectrum}\label{unitspec}%
We require a unit in an algebra to be non-zero.
The notation $e$ is reserved for units.

We denote the unitisation of an algebra $A$ by $\widetilde{A}$.
(If the algebra $A$ is unital, then the unitisation $\widetilde{A}$
is defined to be $A$ itself.)

The spectrum $\s (a) = \s _A (a)$ of an element $a$
of an algebra $A$ is defined as the set of complex numbers $\lambda$
such that $\lambda e - a$ is not invertible in $\widetilde{A}$. The
spectrum in an algebra is the same as in the unitisation.

\subsection{Homomorphisms and Spectra}%
\label{homspec}

Some algebra homomorphism between unital algebras
is called unital if it takes unit to unit.

If $\pi : A \to B$ is an algebra homomorphism, then
\begin{equation}\label{sp}
   \s_B(\pi(a)) \setminus \{0 \} \subset \s_A(a) \setminus \{0 \}
   \quad \text{for all} \quad a \in A,
\end{equation}
\begin{equation}\label{spu}
   \text{if $A$, $B$, $\pi$ are unital, then} \quad
   \s_B(\pi(a)) \subset \s_A(a)  \quad \text{for all} \quad a \in A.
\end{equation}
For a proof, see e.g.\ \cite[Theorem (7.10) p.\ 36]{Thill}.

If $A$ is an algebra, we denote by $\Delta(A)$ the set of multiplicative
linear functionals on $A$. The Gel'fand transform of $a \in A$ is denoted
by $\widehat{a}$.

We equip the spectrum $\Delta(A)$ of an algebra $A$ with the
relative topology induced by the weak* topology.
For an algebra $A$, we have that
\begin{equation}\label{multu}
\widehat{a} \,\bigl( \,\Delta(A) \,\bigr) \subset \s(a)
\quad \text{for all} \quad a \in A,
\end{equation}
cf.\ e.g.\ \cite[Proposition (20.6) p.\ 78]{Thill}, or use
(\ref{sp}) \& (\ref{spu}) above and the fact that
$0 \in \s(a)$ if $A$ is not unital,
as then $A$ is a proper two-sided ideal in $\widetilde{A}$.
That is, ``the range of the Gel'fand transform of $a$
is contained in the spectrum of $a$''.

If $A$ is a \st-algebra, we denote by $\Delta^*(A)$ the
set of Hermitian multiplicative linear functionals on $A$.

\subsection{Hermitian \texorpdfstring{$\ast$}{\9047\073}-Algebras}%
\label{Hermstar}
A \st-algebra is called Hermitian if the spectrum of each of its Hermitian
elements is real.

The above formula (\ref{sp}) also shows that if $\pi$ is a \st-algebra
homomorphism from a Hermitian \st-algebra $A$ to another \st-algebra,
then the range \st-algebra $\pi(A)$ is Hermitian. In particular, if
$\pi$ is a representation of a Hermitian \st-algebra $A$ on a pre-Hilbert
space, then the range $\range(\pi) = \pi(A)$ is a Hermitian \st-algebra.

The above formula (\ref{multu}) implies that a multiplicative linear
functional on a Hermitian \st-algebra $A$ is Hermitian,
$\Delta^* (A) = \Delta (A)$. (Real on Hermitian.)%
\pagebreak

\subsection{Spectral Measures}\label{resid}%

The Borel $\sigma$-algebra is the $\sigma$-algebra
generated by the open sets of a Hausdorff space.
A resolution of identity is a spectral measure $P$ defined on
the Borel $\sigma$-algebra of a Hausdorff space such that the
Borel probability measures $\omega \mapsto \langle P(\omega)x, x \rangle$
are inner regular (and so outer regular) for all unit vectors $x$ in the
Hilbert space on which $P$ acts. (See \cite[Observation (98.5) p.\ 403]{Thill}.)
The support of $P$ then is well-defined, see \cite[Definition (48.38) p.\ 223]{Thill}.
For spectral integrals, cf.\ \cite[\S\ 48 p.\ 213 ff., \S\ 60 p.\ 269 ff.]{Thill}.
Also see \cite[4.3.1 p.\ 74 f., 4.3.2 p.\ 75 ff.]{SchU}.
We also write $P$ for the range of $P$.

For images of spectral measures, see \cite[Definition (48.28) p.\ 220]{Thill}.

\subsection{Commutativity for Unbounded Operators}%
\label{CommUnb}

See \cite[sections 5.5 - 5.6 p.\ 101 ff.]{SchU}.
A possibly unbounded operator $N$ in a Hilbert space is called
normal if it is densely defined, closed and $N^*N = NN^*$,
cf.\ e.g.\ \cite[13.29 p.\ 368]{Rudin}.  It then is maximally normal,
cf.\ e.g.\ \cite[13.32 p.\ 370]{Rudin}.

Let $N$ be a (possibly unbounded) normal operator in a Hilbert space
$\neq \{ 0 \}$. A bounded operator $T$ is said to commute with $N$ if and
only if the domain of $N$ is invariant under $T$, and if $TNx = NTx$ for all
$x$ in the domain of $N$. This is equivalent to $TN \subset NT$. The Spectral
Theorem says that this is the case if and only if $T$ commutes with the
spectral resolution of $N$. See for example \cite[Theorem 13.33 p.\ 371]{Rudin}.

A bounded operator $T$ commutes with a (possibly unbounded) self-adjoint
operator $A$ in a Hilbert space $\neq \{ 0 \}$ if and only if $T$ commutes
with the Cayley transform of $A$, cf. \cite[Lemma (61.15) p.\ 282]{Thill}.

One says that two (possibly unbounded) normal operators in a
Hilbert space $\neq \{ 0 \}$ commute spectrally (or strongly),
if their spectral resolutions commute, cf.\ \cite[Proposition 5.27 p.\ 107]{SchU}.
For two self-adjoint operators this is the case
if and only if their Cayley transforms commute.

A (possibly unbounded) operator $N$ in a Hilbert space $\neq \{ 0 \}$ is
normal if and only if $N = R + \iu S$ where $R$ and $S$ are self-adjoint
and commute spectrally. (See \cite[Proposition 5.30 p.\ 108 f.]{SchU}.)

\section{\textbf{Reminders and Preliminaries}}

\subsection{The Cayley Transformation in Algebras}%
\label{Cayleysubsect}

We next give a few remarks on Cayley transforms in algebras.
(See also \cite[\S\ 15 p.\ 62 f.]{Thill}.)
(The theory of Cayley transforms of densely defined symmetric
operators in a Hilbert space is supposed to be known.
We refer the reader to \cite[7.4A pp.\ 520 - 528]{AMR},
\cite[Chapter 8 pp.\ 229 - 247]{Weid}, \cite[sections VI.13.1-2 pp.\ 283 - 290]{SchU}.)

First, note that an element $a$ of an algebra $A$ commutes with
an invertible element $b$ of $A$ if and only if $a$ commutes with
$b^{\,-1}$.

Next, the Rational Spectral Mapping Theorem says that if $r$
is a non-constant rational functional without pole on the spectrum
of an element $a$ of an algebra, then
\[ \s(r(a)) = r(\s(a)). \]
See for example \cite[Theorem (7.9) p.\ 35]{Thill}.

Let $a$ be an element of an algebra $A$ such that $\s(a)$
does not contain $-i$. Then
\[ g = (a - \iu e ) {^{\,}(a + \iu e)}^{\,-1} \in \widetilde{A} \]
is called the Cayley transform of $a$. The Rational Spectral Mapping
Theorem implies that $1$ is never in the spectrum of a Cayley transform.
(Because the corresponding Moebius transformation does not assume
the value $1$.) In particular $e-g$ is invertible in $\widetilde{A}$.
A simple computation then shows that $a$ can be
regained from $g$ as the inverse Cayley transform
\[ a = \iu \,{(e-g)}^{\,-1\,} (e+g). \]

%
If $A$ is a \st-algebra, and if $a$ is Hermitian, then $g$ is unitary,
i.e.\ $g^*g = gg^* =e$. This follows from the above commutativity
property.

If $A$ is a Hermitian \st-algebra, then the Cayley transformation of a
Hermitian element of $A$ is well-defined. 

\subsection{The Archetypal Spectral Theorem}%
\label{SNAG}%

\begin{reminder}[spectral theorem]\label{specthm}%
Let $C$ be a commutative C*-algebra of bounded linear operators
on a Hilbert space $H \neq \{ 0 \}$ containing the unit operator.
There then exists a unique resolution
of identity $R$ on the spectrum $\Delta(C) = \Delta^*(C)$ with
\[ c = \int \widehat{c} \:\diff R \quad \text{for all} \quad c \in C. \]
One says that $R$ is the spectral resolution of $C$.
A bounded operator on $H$ commutes with the operators
$c \in C$ if and only if it commutes with $R$.
The support of $R$ is all of $\Delta^*(C)$.
(For a proof, see e.g.\ \cite[Theorem 12.22 p.\ 321 f.]{Rudin}
or \cite[Theorem (49.1) p.\ 224]{Thill}.)
\end{reminder}

\section{\textbf{A Sufficient Criterion for Essential Normality}}

If $a$ is a Hermitian element of a \st-algebra $A$ and if either $\iu$ or
$-\iu$ does not belong to $\s(a)$ then neither of them belongs to
$\s(a)$. (Because for an arbitrary element $b$ of $A$, one has
$\s (b^*) = \overline{\s (b)}$ as is easily seen.)%

\begin{proposition}\label{prop}%
   Let $\pi$ be a representation of a \st-algebra $A$ on some pre-Hilbert
   space $H$. Assume that $a$ is a Hermitian element of $A$ whose
   spectrum does not contain either $\iu$ or $-\iu$. Then $\pi (a)$ is essentially
   self-adjoint.
\end{proposition}

\begin{proof}
The operator $\pi (a)$ is symmetric in $H$. It follows from the preceding remark
and formula (\ref{sp}) of subsection \ref{homspec} that the spectrum of $\pi (a)$ in
$\mathrm{End}(H)$ does not contain $\iu$ nor $-\iu$. Hence $\pi (a) \pm \iu \mathds{1}$
are invertible in $\mathrm{End}(H)$, and thus surjective onto $H$. This implies
that $\pi (a)$ is essentially self-adjoint. See \cite[Proposition 7.4.13 p.\ 522]{AMR}.%
\end{proof}

\begin{proposition}\label{spmeas}
   Let $\pi$ be a representation of \st-algebra $A$ on a pre-Hilbert space
   $\neq \{ 0 \}$. Let $a,b$ be commuting Hermitian elements of $A$ each of
   whose spectrum does not contain either $\iu$ or $-\iu$. The self-adjoint closures
   of $\pi (a)$ and $\pi (b)$ then commute spectrally.
\end{proposition}

\begin{proof}
This is so because in this case the Cayley transforms of $\pi (a)$ and
$\pi (b)$ are commuting bijective isometries of the pre-Hilbert space.
\end{proof}

\begin{lemma}
   Let $A, B$ be two essentially self-adjoint operators in a Hilbert space
   $\neq \{ 0 \}$, whose closures $\overline{A}, \overline{B}$ commute
   spectrally. Then $A + \iu B$ is essentially normal and its closure is
   $\overline{A} + \iu \overline{B}$.%
\end{lemma}

\begin{proof}
As noted in subsection \ref{CommUnb}, an operator $N$ is normal if and
only if $N = R + \iu S$ where $R$ and $S$ are self-adjoint and commute
spectrally. This implies that $\overline{A} + \iu \overline{B}$ is normal. In
particular $\overline{A} + \iu \overline{B}$ is closed. It follows that
$A + \iu B$ is closable and its closure is easily seen to extend the closed
operator $\overline{A} + \iu \overline{B} \supset A + \iu B$.
Therefore $\overline{A + \iu B} = \overline{A} + \iu \overline{B}$.
\end{proof}

\begin{theorem}\label{theo}
   Let $\pi$ be a representation of a \st-algebra $A$ on a pre-Hilbert space
   $\neq \{ 0 \}$. Let $a, b$ be commuting Hermitian elements of $A$ each
   of whose spectrum does not contain either $\iu$ or $-\iu$, and put
   $c \vcentcolon= a + \iu b$. 
   The operators $\pi (a)$ and $\pi (b)$ are essentially self-adjoint and their
   closures commute spectrally. The operator $\pi (c)$ is essentially normal,
   and one has
\[ \overline{\pi (c)} = \overline{\vphantom{B}\pi (a)} + \iu \overline{\pi (b)}. \]
\end{theorem}

\begin{proof}
This follows from the preceding three items.
\end{proof}

\begin{corollary}\label{essnormal}
Let $\pi$ be a representation of some commutative Hermitian
\st-algebra $A$ on a pre-Hilbert space $\neq \{ 0 \}$. Then for each element
$c$ of $A$ the operator $\pi (c)$ is essentially normal. If $c = a + \iu b$ with
$a, b$ Hermitian elements of $A$, then the operators $\pi (a)$ and
$\pi (b)$ are essentially self-adjoint, their closures commute spectrally,
and one has
\[ \overline{\pi (c)} = \overline{\vphantom{B}\pi (a)} + \iu \overline{\pi (b)}. \]
\end{corollary}

Things like these have been noted before,
cf.\ \cite[Corollary 8.1.20 p.\ 210 \& Corollary 9.1.4 p.\ 237]{Schm}
or \cite[Corollary 4.12 p.\ 67 \& Theorem 7.11 p.\ 143]{SchI}.
I came to them through a good question.%

\section{\textbf{Proof of the Main Result}}%
\label{mainsect}

Let $\pi(A)$ be a Hermitian commutative unital \st-algebra of
linear operators on a pre-Hilbert space $H \neq \{ 0 \}$, such that
$\slangle \pi(a)^* x, y \srangle = \slangle x, \pi(a) y \srangle$
holds for all $\pi(a) \in \pi(A)$ and $x, y \in H$.
Then $\Delta\bigl(\pi(A)\bigr) = \Delta^*\bigl(\pi(A)\bigr)$.
Cf.\ \ref{Hermstar}.

Now let $\pi(a) \in \pi(A)$. The linear operator $\pi(a)$ is essentially normal.
To see this, apply corollary \ref{essnormal} to the identical representation of $\pi(A)$.
We shall identify $\pi(a)$ with its normal closure $\overline{\pi(a)}$.
The operators $\overline{\pi(a)}$ then form a \st-algebra.
Let $\pi(A)\sa$ denote the set of Hermitian elements of $\pi(A)$.

Let $\kappa : a \to (a - \iu e ) {^{\,}(a + \iu e)}^{\,-1}$ denote the Cayley transformation
in the Hermitian \st-algebra $\pi(A)$, or in the complex plane, cf.\ subsection \ref{Cayleysubsect}.%

The unitary Cayley transforms $\kappa\bigl(\overline{\pi(a)}\bigr) \in \pi(A)$ of the self-adjoint operators
$\overline{\pi(a)} \in \pi(A)\sa$ generate a unital commutative \st-algebra of bounded operators 
$B$ that is $\subset \pi(A)$, as $\pi(A)$ is Hermitian.
Let $R$ be the spectral resolution of the completion $C$ of $B$,
cf.\ reminder \ref{specthm} in subsection \ref{SNAG}.

We shall map $\Delta^*(C)$ to $\Delta^*\bigl(\pi(A)\bigr)$. Let $\tau \in \Delta^*(C)$. Then
$\tau |_B \in \Delta^*(B)$ as $B$ is unital in $C$. We next use that $B \subset \pi(A)$, as noted above.

\begin{statement}\label{extends}%
The function $\tau |_B \in \Delta^*(B)$ extends uniquely to $\sigma \in \Delta^*\bigl(\pi(A)\bigr)$.
\end{statement}

\begin{proof}
The set
\[ S \vcentcolon= \{ \,c \in B : c \text{ is invertible in } \pi(A) \,\} \]
is a unital \st-subsemigroup of $B$. Hence the set of fractions
\[ F : = \{ \,c^{\,-1\,} b \in \pi(A) : c \in S, \ b \in B \,\} \]
forms a unital \st-subalgebra of $\pi(A)$.

We shall prove that $F$ is all of $\pi(A)$.
Since $F$ is a complex vector space,
it suffices to prove that $F$ contains every element of $\pi(A)\sa$.
Let $\pi(a) \in \pi(A)\sa$, and let $g \in B$ be its Cayley transform.
Cf.\ subsection \ref{Cayleysubsect}. Then $\pi(a)$ is the inverse Cayley transform of $g$:
\[ \pi(a) = \mathrm{i} \,{(e-g)}^{\,-1} (e+g). \]
This says that $\pi(a)$ is of the form ${c}^{\,-1\,} b$ where $b, c \in B$
with $c$ invertible in $\pi(A)$. That is, $\pi(a) \in F$. Thus $F$ contains
every element of $\pi(A)\sa$, as was to be shown.

For $\pi(a) = c^{\,-1\,} b \in \pi(A)$ with $c \in S$, $b \in B$,
put $\sigma\bigl(\pi(a)\bigr) \vcentcolon= {\tau(c)}^{\,-1} \,\tau(b)$.
Please note here that $\tau(c) \in \s_{\pi(A)}(c)$ is never zero for
$c \in S \subset B \subset \pi(A)$,
cf.\ formula (\ref{multu}) of subsection \ref{homspec}.
Please note that this definition is independent of the
representatives $c \in S$, $b \in B$ of $\pi(a) \in \pi(A)$.
It is easy to verify that this definition extends $\tau|_B$ uniquely to a
Hermitian multiplicative linear functional $\sigma$ on all of $\pi(A)$.
\end{proof}

\begin{statement}\label{function}%
For $\pi(a) \in \pi(A)\sa$ and $\tau \in \Delta^*(C)$ we have
\[ \wht{\kappa\bigl(\overline{\pi(a)}\bigr)}(\tau) = \kappa\bigl( \wht{\pi(a)}(\sigma) \bigr). \]
\end{statement}

\begin{proof}
Indeed, by the homomorphic nature of $\sigma \in \Delta^*\bigl(\pi(A)\bigr)$, we then find
\[ \tau \Bigl( \kappa \bigl( \overline{\pi(a)} \bigr) \Bigr)
= \sigma \Bigl( \kappa \bigl( \overline{\pi(a)} \bigr) \Bigr)
= \kappa \Bigl( \sigma \bigl( \pi(a) \bigr) \Bigr). \qedhere \]
\end{proof}

A spectral measure $Q$ on a subset of $\Delta^*\bigl(\pi(A)\bigr)$ is defined as the
image measure $f(R)$ of $R$ under the above injective map
$f : \Delta^*(C) \to \Delta^*\bigl(\pi(A)\bigr)$, see section \ref{resid}.

\begin{statement}\label{statp1}%
The spectral measure $Q$ satisfies
\[ \overline{\pi(a)} = \int \wht{\pi(a)} \,\diff Q \quad \text{for all} \quad \pi(a) \in \pi(A). \]
\end{statement}

\begin{proof}
For $\pi(a) \in \pi(A)\sa$, we compute
\begin{align*}
\kappa\bigl(\overline{\pi(a)}\bigr) & = \int \wht{\kappa\bigl(\overline{\pi(a)}\bigr)} \,\diff R \\
& = \int \kappa\bigl( \wht{\pi(a)} \bigr) \,\diff Q \tag*{see statement \ref{function}} \\
& = \int \kappa \,\diff \mspace{2mu} \wht{\pi(a)} (Q) \tag*{see \cite[(60.10) p.\ 272]{Thill}} \\
& = \kappa\mspace{2mu}\bigl( \int \id \,\diff \mspace{2mu} \wht{\pi(a)} (Q) \bigr)
\tag*{see \cite[(61.18) p.\ 283]{Thill}} \\
& = \kappa\mspace{2mu}\bigl( \int \wht{\pi(a)} \,\diff Q \bigr) \tag*{see again \cite[(60.10) p.\ 272] {Thill}}
\intertext{or, by injectivity of the Cayley transformation}
\overline{\pi(a)} & = \int \wht{\pi(a)} \,\diff Q.
\end{align*}
Let now $\pi(c) \in \pi(A)$ be arbitrary, and let $\pi(a), \pi(b)$ be Hermitian
elements of $\pi(A)$ with $\pi(c) = \pi(a) + \iu \pi(b)$. Corollary \ref{essnormal}
(applied to the identical representation of $\pi(A)$)
says that $\overline{\pi(c)}$ is normal and that
\begin{align*}
\overline{\pi(c)} = \overline{\pi(b)} + \iu \overline{\pi(a)}
 & = \int \widehat{\pi(a)} \,\diff Q
 + \iu \int \widehat{\pi(b)} \,\diff Q \\
 & \subset \int \widehat{\pi(c)} \,\diff Q,
\end{align*}
cf.\ \cite[Theorem (60.14) p.\ 273]{Thill}.
Since the member on the right hand side is normal,
cf.\ \cite[Theorem (60.23) p.\ 275]{Thill},
we get by maximal normality of $\overline{\pi(c)}$
(subsection \ref{CommUnb}) that
\[ \overline{\pi(c)} = \int \widehat{\pi(c)} \,\diff Q. \qedhere \]
\end{proof}

Let now $\pi$ be a unital representation of a unital \st-algebra $A$ on a pre-Hilbert space
$H \neq \{ 0 \}$. Assume that its image $\pi(A)$ is commutative and Hermitian,
e.g.\ $A$ commutative and Hermitian.
Please note that then $\Delta^*\bigl(\pi(A)\bigr) = \Delta\bigl(\pi(A)\bigr)$.
Continue with the notation as above.

Consider the map $\pi^*$ adjoint to $\pi$, given by
\begin{align*}
\pi^* : \Delta^*\bigl(\pi(A)\bigr) & \to \Delta^*(A) \\
                           \sigma & \mapsto \pi^*(\sigma) = \sigma \circ \pi.
\end{align*}
This is well-defined as the \st-algebras are unital.
So a spectral measure $P$ on a subset of $\Delta^*(A)$ is defined as the image
$\pi^*(Q)$ of $Q$ under the injective map $\pi^*$.%

\begin{statement}
The spectral measure $P$ satisfies
\[ \overline{\pi(a)} = \int \wht{a} \,\diff P \quad \text{for all} \quad a \in A. \]
\end{statement}

\begin{proof}
We compute
\begin{align*}
\overline{\pi(a)} & = \int \wht{\pi(a)} \,\diff Q \\
 & = \int \sigma\bigl(\pi(a)\bigr) \,\diff Q(\sigma) \\
 & = \int (\sigma \circ \pi) (a) \,\diff Q(\sigma) \\
 & = \int \wht{a}\mspace{2mu}(\sigma \circ \pi) \,\diff Q(\sigma) \\
 & = \int \wht{a} \bigl( \pi^*(\sigma) \bigr) \,\diff Q(\sigma) \\
 & = \int (\wht{a} \circ \pi^*) (\sigma) \,\diff Q(\sigma) \\
 & = \int (\wht{a} \circ \pi^*) \,\diff Q \\
 & = \int \wht{a} \,\diff \pi^*(Q) = \int \wht{a} \,\diff P \pagebreak \qedhere
\end{align*}
\end{proof}

Statements (ii) and (iii) of the main result, theorem \ref{main}, follow as in
\cite[Proposition (49.11) p.\ 229]{Thill}. In order to prove statement (iv),
we note that if a bounded linear operator $b$ on the completion of $H$
commutes with $P$, then it commutes with all operators $\overline{\pi(a)}$
$(a \in A)$, by statement (i). To prove the converse, we note that if a
bounded linear operator on the completion of $H$ commutes with all
operators $\overline{\pi(a)}$ $(a \in A)$, then it commutes with all operators
in the commutative C*-algebra $C$, by subsection \ref{CommUnb},
and so with its spectral resolution $R$, and then with its image $Q$,
and thus with the image $P$ thereof.
This finishes the proof of the main result, theorem \ref{main}.


\section{\textbf{Appendix: Derivation of Some Known Consequences}}

\begin{reminder}[the GNS construction]\label{GNS}%
Let $A$ be a unital \st-algebra, and let $\varphi$ be a
positive linear functional on $A$. (That is, $\varphi (a^*a) \geq 0$
for all $a \in A$.) Assume that $\varphi$ is non-zero. One defines
a positive semidefinite sesquilinear form on $A$ by putting
\[ \langle a, b \rangle = \varphi (b^*a) \qquad (a, b \in A). \]
This sesquilinear form is a so-called Hilbert form, in that it satisfies
\[ \langle ab, c \rangle = \langle b, a^*c \rangle
\quad \text{for all} \quad a, b, c \in A. \]
It follows that the isotropic subspace
\[ I \vcentcolon= \{ \,a \in A : \langle a, b \rangle = 0 \text{ for all } b \in B \,\} \]
is a left ideal in $A$. The quotient space $H \vcentcolon= A / I$ is a
pre-Hilbert space when equipped with the inner product
\[ \langle a + I , b + I \rangle = \langle a, b \rangle \qquad (a, b \in A). \]
(This follows from the Cauchy-Schwarz inequality.) The assignment
\[  \pi(a) ( b + I ) \vcentcolon= ab + I \qquad (a, b \in A) \]
defines a unital representation $\pi$ of $A$ on the pre-Hilbert space
$H \neq \{ 0 \}$. (This uses the fact that the sesquilinear form is a Hilbert form.)
The vector $e+I$ is (algebraically) cyclic for the representation $\pi$, and the positive linear
functional $\varphi$ can be regained from the cyclic representation $\pi$ via
\[ \varphi (a) = \langle \pi(a) ( e + I ), ( e + I ) \rangle
\quad \text{for all} \quad a \in A. \]
\end{reminder}

\begin{reminder}[semiperfect]%
A commutative unital \st-algebra $A$ is said to be \emph{semiperfect}, if every positive
linear functional $\varphi$ on $A$ with $\varphi(e) = 1$ has a representing
measure $\mu$ in the sense that
\[ \varphi(a) = \int \widehat{a} \,d \mu \quad \text{for all} \quad a \in A \]
for some probability measure $\mu$ on a subset of $\Delta^*(A)$.%
\end{reminder}

\begin{theorem}\label{semiper}%
A Hermitian commutative unital \st-algebra is semiperfect.%
\end{theorem}

\begin{proof}
Let $\varphi$ be a positive linear functional on a Hermitian commutative
unital \st-algebra $A$ such that $\varphi(e) = 1$. Consider the cyclic unital
representation $\pi$ associated to $\varphi$ by the GNS construction.
Let $P$ be the spectral measure provided by Theorem \ref{main}.
Now $\varphi$ can be regained from $\pi$ via
\[ \varphi (a) = \langle \pi(a) ( e + I ), ( e + I ) \rangle
\quad \text{for all} \quad a \in A. \]
We may then take the probability measure given by
\[ \mu : \Delta \mapsto \langle P(\Delta) ( e + I ), ( e + I ) \rangle \]
as a representing measure.%
\end{proof}

\begin{lemma}\label{field}%
   A \st-algebra, which is a field different from $\mathds{C}$,
   carries no positive linear functionals different from zero.
   (This applies for example to the field $\mathds{C}(x)$ of
   rational functions endowed with the involution making
   the variable $x$ Hermitian, cf.\ \cite{Len}.)
\end{lemma}

\begin{proof}
Let $A$ be a \st-algebra, which is a field different from $\mathds{C}$.
We first note that the elements of $A \setminus \mathds{C}e$ have empty
spectrum because $A$ is a field. It follows that $A$ is Hermitian. It also
follows that there is no multiplicative linear functional on $A$. Indeed,
for $a \in A \setminus \mathds{C}e$, we would run into a contradiction
with formula (\ref{multu}) of subsection \ref{homspec}, by the fact that
$\s (a)$ is empty as noted above. The statement follows now from the
semiperfectness of $A$ guaranteed by the preceding theorem \ref{semiper}.%
\end{proof}

\begin{theorem}\label{zerorep}%
   A \st-algebra, which is a field different from $\mathds{C}$,
   does not have any representation on pre-Hilbert spaces
   other than the zero representations.
   (This applies for example to the field $\mathds{C}(x)$ of
   rational functions endowed with the involution making
   the variable $x$ Hermitian.)
\end{theorem}

\begin{corollary}[See \hbox{\cite[Theorem 2.1.12 p.\ 38]{Schm}} for a stronger result]%
A \st-algebra of operators on a pre-Hilbert space which is
a field consists of scalar multiples of the identity operator.
\end{corollary}

\begin{reminder}[the radical]\label{radical}%
Let $A$ be an algebra. The radical $\mathrm{rad}(A)$ of $A$ is the
set of those elements $a$ of $A$ such that $e+ba$ is invertible in the
unitisation of $A$ for all $b \in A$ cf.\ \cite[Proposition III.24.16 (ii) p.\ 125]{BD}.
Please note that this definition of the radical can be given in terms of the
spectrum in the algebra.

An element $a$ of $\mathrm{rad}(A)$ satisfies
$\s_A (a) \subset \{ 0 \}$, \cite[Proposition III.24.16 (i) p.\ 125]{BD}.

An algebra is called radical if it coincides with its own radical.
For example, the radical of an algebra is a radical algebra.
(This follows from \cite[Corollary III.24.20 p.\ 126]{BD}.)
\end{reminder}

We get a conceptually easy proof of an essential theorem of Torben Maack Bisgaard
\ref{Bis} below.

\begin{observation}[see e.g.\ \hbox{\cite[Proposition I.3.62 p.\ 95]{Hel}}]\label{Helem}%
A multiplicative linear functional on an algebra vanishes on the radical.
\end{observation}

\begin{proof}
Indeed, we have
\[ \tau(a) \in \s(a) \subset \{ 0 \} \]
for any multiplicative linear functional $\tau$
and any $a$ in the radical, as can be seen from
the second to last statement of the preceding reminder
\ref{radical} and from formula (\ref{multu}) from subsection \ref{homspec}.
\end{proof}

\begin{observation}\label{radHerm}%
The radical of a \st-algebra is a Hermitian \st-algebra.
\end{observation}

\begin{proof}
Let $A$ be a \st-algebra.
It is easily seen from the definition that the radical
$\mathrm{rad}(A)$ is a \st-subalgebra of $A$,
using the involution on the definition as in reminder
\ref{radical}, as well as the elementary facts that
$\s(a^*) = \overline{\s(a)}$ for any $a \in A$
and that $\s(ab) \setminus \{ \,0 \,\} = \s(ba) \setminus \{ \,0 \,\}$,
cf. \ \cite[Proposition (7.12) p.\ 37]{Thill}.
Using the fact that $\mathrm{rad}(A)$ is a radical algebra,
we find that an element $a$ of $\mathrm{rad}(A)$ satisfies
$\s_{\mathrm{rad}(A)} (a) \subset \{ 0 \}$, as is seen from
the preceding reminder \ref{radical}. This implies in
particular that the radical $\mathrm{rad}(A)$ is Hermitian.
\end{proof}

\begin{corollary}\label{cor}%
A positive linear functional on some commutative unital \st-algebra vanishes on the radical.
\end{corollary}

\begin{proof}
Consider a commutative unital \st-algebra $A$.
Then $\mathds{C}e + \mathrm{rad}(A)$ is a unital Hermitian
\st-algebra by the preceding observation \ref{radHerm},
using that the spectrum in an algebra is the same
as in the unitisation, cf.\ section \ref{unitspec}.
Thus $\mathds{C}e + \mathrm{rad}(A)$ is semiperfect
by the above theorem \ref{semiper}.
The statement follows now from observation \ref{Helem}.
\end{proof}

The assumption of commutativity can be dropped, by considering a suitable
commutative \st-subalgebra, as we shall do in the following proof.

\begin{theorem}[Torben Maack Bisgaard \cite{TMB}]\label{Bis}%
A positive linear functional on a unital \st-algebra vanishes on the radical.
\end{theorem}

\begin{proof}
Let $A$ be a unital \st-algebra, let $\varphi$ be a positive linear functional
on $A$, and let $c \in \mathrm{rad}(A)$. To prove that $\varphi (c) = 0$, it
suffices to show that $\varphi (c^*c) = 0$ by the Cauchy-Schwarz inequality
$| \,\varphi (c) \,| \leq \varphi (e) \varphi (c^*c)$. Please note that
$c^*c \in \mathrm{rad}(A)$ as well. The second commutant $B$ of
$c^*c$ in $A$ is a commutative unital \st-subalgebra of $A$, cf.\
\cite[Observation (19.6) p.\ 74]{Thill}, such that $\s_B (b) = \s_A (b)$
for all $b \in B$, cf.\ \cite[Lemma (19.8) p.\ 74]{Thill}. It follows that
$\mathrm{rad}(A) \cap B \subset \mathrm{rad}(B)$, by the possibility to define
the radical through the spectrum. In particular $c^*c \in \mathrm{rad}(B)$.
The statement follows now from the preceding theorem \ref{cor}
applied to the commutative unital \st-algebra $B$.%
\end{proof}

\begin{reminder}[extensibility]\label{extens}%
Let $A$ be a \st-algebra. A positive linear functional on $A$
is called \emph{extensible}, if it can be extended to a positive linear
functional on the unitisation of $A$. Cf.\ \cite[Proposition (35.3) p.\ 160 f.]{Thill}.
\end{reminder}

For example, if $\pi$ is a representation of $A$ on a
pre-Hilbert space $H$, then the positive linear functionals
given by
\[ \varphi(a) \vcentcolon= \slangle \pi(a) x, x \srangle \qquad (a \in A) \]
are extensible for all vectors $x \in H$, by considering
a representation of the unitisation of $A$ extending $\pi$.

\begin{corollary}
An extensible positive linear functional on a \st-algebra
vanishes on the radical.
\end{corollary}

\begin{proof}
Consider the unitisation $\widetilde{A}$ of a \st-algebra $A$,
and use that the radical of $A$ is contained in (and thus coincides with)
the radical of $\widetilde{A}$, by applying the definition of the radical
as in reminder \ref{radical}, and the fact that $\mathrm{rad}(\widetilde{A})$
is a subspace of $\widetilde{A}$ and does not contain the unit, as
the spectrum of an element in a radical consists of zero at most,
cf.\ reminder \ref{radical}.
\end{proof}

\begin{theorem}\label{repvanrad}%
A representation of a \st-algebra on a pre-Hilbert space
vanishes on the radical.%
\end{theorem}

\begin{corollary}%
   A radical \st-algebra
   does not have any representation on
   pre-Hilbert spaces other than the zero representations.
\end{corollary}

This applies for instance to convolution \st-algebras
on an interval of the positive half-line.
Cf.\ \cite[vol.\ I, \S\ 2.IV.15 p.\ 179 ff.; vol.\ II, \S\ 6.V.1 p.\ 156 f.]{Mik}
or \cite[\S\ 6.V.1 p.\ 369 f.]{Mikold}.

\begin{reminder}[semisimple]%
An algebra is called \emph{semisimple} if its radical consists of $0$ alone.
\end{reminder}

\begin{reminder}[faithful]%
A representation of a \st-algebra on a pre-Hilbert space
is called \emph{faithful} if it is injective.
\end{reminder}

\begin{corollary}
A \st-algebra that admits a faithful representation on a pre-Hilbert space
is semisimple, by theorem \ref{repvanrad}.%
\pagebreak
\end{corollary}

\begin{corollary}[see \hbox{\cite[Corollary 7.24 p.\ 150]{SchI}}]%
If a Hermitian commutative unital \st-algebra $A$ admits a faithful unital
representation $\pi$ on a pre-Hilbert space, then the set of Hermitian
characters $\Delta^*(A)$ separates the points of $A$.%
\end{corollary}

\begin{proof}
Let $a \in A$. Then $\pi(a) \neq 0$ as $\pi$ is faithful.
So there exists $\tau \in \Delta^*(A)$ with $\tau(a) \neq 0$,
by the main result \ref{main}.%
\end{proof}

The \st-radical is defined as the intersection of all representations
of a \linebreak \st-algebra on pre-Hilbert spaces.
To make that rigorous, cf.\ Schm\"{u}dgen \cite[p.\ 81 f.]{SchI}.
Theorem \ref{repvanrad} says that the radical
of any \st-algebra is contained in the \st-radical.
This is of concern to factoring over the radical.


\end{document}